\documentclass[11 pt, reqno]{amsart}

\numberwithin{equation}{section} \theoremstyle{plain}
\usepackage{amssymb,amsmath,amsfonts,mathrsfs}
\usepackage{appendix}
\usepackage{graphicx}
\usepackage{color}
\usepackage[colorlinks,citecolor=blue,filecolor=black,linkcolor=blue,urlcolor=black]{hyperref}
\usepackage{enumitem,cleveref}
\usepackage[reqno]{amsmath}

\makeatletter
\newcommand{\reqnomode}{\tagsleft@true}
\makeatother

\newtheorem{theorem}{Theorem}[section]
\newtheorem{lemma}[theorem]{Lemma}
\newtheorem{cor}[theorem]{Corollary}

\theoremstyle{definition}

\theoremstyle{remark}
\newtheorem{remark}[theorem]{Remark}

\numberwithin{equation}{section}

\newcommand{\loc}{\operatorname{loc}}
\newcommand{\comp}{\operatorname{comp}}

\newcommand{\End}{\operatorname{End}}

\newcommand{\RE}{\operatorname{Re}}

\newcommand\CC{\mathbb{C}}
\newcommand\RR{\mathbb{R}}

\newcommand{\mcomp}{C_{c}^{\infty}(M)}

\newcommand{\Del}{\Delta}

\renewcommand{\div}{\operatorname{div}}

\newcommand{\ve}{\varepsilon}

\newcommand{\vbn}{\mathcal{E}}
\newcommand{\vcomp}{\Gamma_{C_{c}^{\infty}}(\mathcal{E})}

\newcommand{\vep}{\varepsilon}

\begin{document}
\title[Covariant Schr\"odinger Operator]{Covariant Schr\"odinger Operator and $L^2$-Vanishing Property on Riemannian Manifolds}
\author{Ognjen Milatovic}
\address{Department of Mathematics and Statistics\\
         University of North Florida\\
       Jacksonville, FL 32224 \\
        USA}
\email{omilatov@unf.edu}



\begin{abstract} Let $M$ be a complete Riemannian manifold satisfying a weighted Poincar\'e inequality, and let $\mathcal{E}$ be a Hermitian vector bundle over $M$ equipped with a metric covariant derivative $\nabla$. We consider the operator
$H_{X,V}=\nabla^{\dagger}\nabla+\nabla_{X}+ V$, where $\nabla^{\dagger}$ is the formal adjoint of $\nabla$ with respect to the inner product in the space of square-integrable sections of $\mathcal{E}$, $X$ is a smooth (real) vector field on $M$, and $V$ is a fiberwise self-adjoint, smooth section of the endomorphism bundle $\End\vbn$. We give a sufficient condition for the triviality of the $L^2$-kernel of $H_{X,V}$. As a corollary, putting $X\equiv 0$ and working in the setting of a Clifford module equipped with a Clifford connection $\nabla$, we obtain the triviality of the $L^2$-kernel of $D^2$, where $D$ is the Dirac operator corresponding to $\nabla$. In particular, when $\vbn=\Lambda_{\CC}^{k}T^*M$ and $D^2$ is the Hodge--deRham Laplacian on (complex-valued) $k$-forms, we recover some recent vanishing results for $L^2$-harmonic (complex-valued) $k$-forms.
\end{abstract}

\maketitle

\section{Introduction}
For many years mathematicians have studied the triviality property of the space $\mathscr{K}_{\Delta}$ of $L^2$-harmonic $k$-forms on complete Riemannian manifolds without boundary,
\[
\mathscr{K}_{\Delta}:=\{\omega \in L^2\colon\Delta \omega=0\},
\]
where $\Delta:=d\delta+\delta d$ is the Hodge--deRham Laplacian acting $k$-forms (here, $d$ and $\delta$ are the standard differential and codifferential).

Topological significance of $\mathscr{K}_{\Delta}=\{0\}$ on a compact Riemannian manifold $M$ is clear if we remember that the space $\mathscr{K}_{\Delta}$ is isomorphic to the $k$-th de Rham cohomology group of $M$. While this isomorphism is generally not present in the setting of a non-compact Riemannian manifold $M$, it turns out that the triviality of $\mathscr{K}_{\Delta}$  may still offer some topological insights: for example, the authors of~\cite{Li-Tam-92} showed that if $M$ has no parabolic ends and $\mathscr{K}_{\Delta}=\{0\}$, where $\Delta$ is Hodge--deRham Laplacian acting on $1$-forms, then $M$ is connected at infinity.

Some forty years ago, the author of~\cite{Dod-81} introduced an elegant method for tackling the problem of triviality of $\mathscr{K}_{\Delta}$ pertaining to $k$-forms on a complete Riemannian manifold $M$. Using Lichnerowicz--Weitzenb\"ock formula (see~(\ref{E:Weitzenbock}) below) and a suitable sequence of cut-off functions (whose existence is guaranteed by the completeness of $M$; see section~\ref{SS:cut-off} below for details), the author of~\cite{Dod-81} showed, among other things, that if the volume of $M$ is infinite and the Weitzenb\"ock curvature operator $\mathscr{R}^{W}$ is non-negative definite, then $\mathscr{K}_{\Delta}=\{0\}$, generalizing an earlier result of~\cite{Yau-76} pertaining to $1$-forms.

In subsequent years, a number of authors have refined the integration by parts technique from~\cite{Dod-81}, aiming to accommodate various assumptions on $M$ and $\mathscr{R}^{W}$ (in the case of $1$-forms, $\mathscr{R}^{W}$ reduces to Ricci tensor $\textrm{Ric}_{M}$). About twenty years ago, in the context of $1$-forms, the authors of~\cite{Li-Wang-01}, showed that if $M$ (with $\dim M=n$) satisfies $\lambda_1(M)>0$ and $\textrm{Ric}_{M}\geq-\frac{n\lambda_1(M)}{n-1}+\varepsilon$, for some $\varepsilon>0$, then $\mathscr{K}_{\Delta}=\{0\}$ (see~(\ref{E:first-ev}) below for the definition of the bottom of the spectrum $\lambda_1(M)$ of the scalar Laplacian $\Del_{M}$). Later, this result was generalized in~\cite{Lam-10} to manifolds satisfying Poincar\'e inequality with a (continuous) weight $\rho\geq 0$:
\begin{equation}\label{E:poincare-intro}
\int_{M}\rho(x)|f(x)|^2\,d\nu_{g}(x)\leq \int_{M}|df(x)|^2d\nu_{g}(x),
\end{equation}
for all $f\in\mcomp$, where $\mcomp$ denotes smooth compactly supported functions on $M$ and $d\nu_{g}$ is the volume element on $M$ induced by the metric $g$. The author of~\cite{Lam-10} imposed a certain condition on the growth of $\rho$ and the following condition on the Ricci tensor: $\textrm{Ric}_{M}\geq-\frac{n\rho}{n-1}+\varepsilon$, for some $\varepsilon>0$.

Subsequent to~\cite{Lam-10}, the author of~\cite{Vieira-16} proved (see theorem 5 there) that $\mathscr{K}_{\Delta}=\{0\}$ for $k$-forms under the following assumptions: $M$ satisfies~(\ref{E:poincare-intro}) (without growth-rate or sign restrictions on $\rho$), $M$ has infinite volume or $\rho$ is not identically equal to $0$, and Weitzenb\"ock curvature operator satisfies $\mathscr{R}^{W}\geq -a \rho$, where $a\in [0,a_0)$ is a constant (here, the constant $a_0$ comes from the refined Kato inequality for $k$-forms). A related vanishing result for $L^{q}$-harmonic $(0,k)$-tensors with $q\geq 2$ (here, ``harmonic" is meant with respect to the Lichnerowicz Laplacian) was established by the authors of~\cite{GHKC-21} under the following assumptions: $M$ satisfies~(\ref{E:poincare-intro}), $M$ is non-parabolic,  $\displaystyle\liminf_{x\to\infty}\rho(x)>0$, and the curvature condition $\mathscr{C}\geq -a \rho$, where $a\in [0,a_0)$, with $a_0$ depending (among other things) on $q$. (Here, $\mathscr{C}$ is a suitable curvature operator; see section 2 in~\cite{GHKC-21} for details.) The paper~\cite{GHKC-21} (see also~\cite{GHKC-22-kah} for the K\"ahler manifold setting) gives a number of vanishing results (for $(0,k)$-tensors and $k$-forms) in which the requirement $\mathscr{C}\geq -a \rho$ is replaced by more explicit conditions involving eigenvalues of $\mathscr{C}$.

By performing a careful analysis of the Weitzenb\"ock curvature operator, the author of~\cite{Lin-19} established two types of vanishing results for $k$-forms: (i) theorems based on integral-type assumptions on the Weyl curvature tensor $W$, traceless Ricci tensor $E$, and scalar curvature, and (ii) theorems based on the assumption~(\ref{E:poincare-intro}) and pointwise assumptions on $W$ and $E$.

Over the last fifteen years, some authors have studied vanishing property assuming weighted Poincar\'e inequality for $k$-forms (with continuous weight $\rho\geq 0$):
\begin{equation}\label{E:poincare-intro-form}
\int_{M}\rho(x)|\omega(x)|^2\,d\nu_{g}(x)\leq \int_{M}(|d\omega(x)|^2+|\delta\omega(x)|^2)\,d\nu_{g}(x),
\end{equation}
for all smooth compactly supported $k$-forms $\omega$.

Assuming~(\ref{E:poincare-intro-form}) with some growth restrictions on $\rho$, the vanishing property for harmonic $k$-forms was established in~\cite{Chen-Sung-09} and, subsequently, in~\cite{Dung-Sung-14}.  Recently, the author of~\cite{Zhou-20} proved (see theorem 1.4 there) that $\mathscr{K}_{\Delta}=\{0\}$ under the following assumptions: $M$ satisfies~(\ref{E:poincare-intro-form}), $\rho$ is not identically equal to $0$, and Weitzenb\"ock curvature operator satisfies $\mathscr{R}^{W}\geq -a \rho$,  where $a\geq 0$ is a constant. Additionally, the authors of~\cite{Duc-Dang-21} established several vanishing theorems for $k$-forms assuming~(\ref{E:poincare-intro-form}) together with pointwise conditions on Weyl conformal curvature tensor and traceless Ricci tensor.

As can be seen from the preceding paragraphs, in recent years there has been quite a bit of activity on the $L^2$-vanishing property for harmonic $k$-forms. For the corresponding studies in the setting of $p$-harmonic $k$-forms, we refer the reader to~\cite{Chao-22,Dung-Dung-N-23,Dung-17,Dung-22} and references therein.  (Here, a $k$-form $\omega$ is $p$-harmonic, $p>1$, if $d\omega=0$ and $\delta(|\omega|^{p-2}\omega)=0$.) For $L^{q}_{f}$-vanishing results (in some papers $q=2$) in the context of $1$-forms on smooth metric-measure spaces (Riemannian manifolds $(M,g)$ with metric $g$ and measure $e^{-f}\,d\nu_{g}$, where $f$ is a smooth function on $M$ and $d\nu_{g}$ is the volume measure induced by the metric $g$), see~\cite{Chao-23,Han-Lin-17,Vieira-13,Zhou-21,Zhou-lp} and references therein.

Before describing the results of our article, we note the paper~\cite{C-99}, which is situated in the setting of a Hermitian vector bundle $\vbn$ (over a complete Riemannian manifold $M$), equipped with a metric covariant derivative $\nabla$; see section~\ref{SS:ornst-uhl-def} below for details. Denoting by $\Gamma_{L^2}(\vbn)$ the square integrable sections and by $\nabla^{\dagger}$ the formal adjoint of $\nabla$ (with respect to the inner product in $\Gamma_{L^2}(\vbn)$), the author of~\cite{C-99} considered the \emph{covariant Schr\"odinger operator}
\[
H_{V}=\nabla^{\dagger}\nabla+V,
\]
where  $V$ is a fiberwise self-adjoint, smooth section of the endomorphism bundle $\End\vbn$.
Let us the denote the $L^2$-kernel of $H_{V}$ by
\begin{equation}\label{E:kern-h-v}
\mathscr{K}_{H_{V}}:=\{u\in \Gamma_{L^2}(\vbn)\colon H_{V}u=0\},
\end{equation}
where  $H_{V}u=0$ is understood in distributional sense (as in~(\ref{E:KER-H-V}) below).

In the paper~\cite{C-99} the author observed that in the case $\vbn=\Lambda_{\CC}^{k}T^*M$ (the complexified version of $\Lambda^{k}T^*M$),  Lichnerowicz--Weitzenb\"ock formula (see~(\ref{E:Weitzenbock}) below) leads to the following equality: $\mathscr{K}_{\Delta}=\mathscr{K}_{H_{V}}$, where $V=\mathscr{R}^{W}$ and $\Delta$ is the Hodge--deRham Laplacian on $k$-forms (here, $\mathscr{R}^{W}$ is the Weitzenb\"ock curvature operator). Thus, establishing the $L^2$-vanishing property for $k$-forms amounts to proving that $\mathscr{K}_{H_{V}}=\{0\}$. In particular, the author of~\cite{C-99} showed that $\mathscr{K}_{H_{V}}=\{0\}$ provided that $M$ satisfies a Sobolev $p$-type inequality with $p>2$, and that $|V_{-}|$ satisfies a certain integral-type condition (here, $V_{-}$ is the negative part of $V$). In the recent years, the operator $H_{V}$ has been studied extensively; see the book~\cite{Gue-book}.

In our article we consider the operator
\[
H_{X,V}=\nabla^{\dagger}\nabla+\nabla_{X}+ V.
\]
Here, $\nabla$ and $\nabla^{\dagger}$ are as in the preceding paragraph, $X$ is a real, smooth (generally unbounded) vector field on $M$, and $V$ is a fiberwise self-adjoint, smooth section of the endomorphism bundle $\End\vbn$. We define $\mathscr{K}_{H_{X,V}}$ as in~(\ref{E:kern-h-v}) with $H_{X,V}$ in place of $H_{V}$.

In theorem~\ref{T:main-1} we prove that $\mathscr{K}_{H_{X,V}}=\{0\}$ under the following assumptions: $M$ satisfies~(\ref{E:poincare-intro}), $M$ has infinite volume or $\rho$ is not identically equal to zero, $|X|\leq\hat{a}\sqrt{\rho}$ and $V-\div X\geq -a\rho$, with constants $0\leq \hat{a}<1$ and $0\leq a<1-\hat{a}$ (here, $\div X$ is the divergence of $X$). As a corollary, putting $X\equiv 0$ and working in the setting of a Clifford module $\vbn$ equipped with a Clifford connection $\nabla$ (see section~\ref{SS:cb} below for details) we get $\mathscr{K}_{D^2}=\{0\}$, where $D$ is the Dirac operator corresponding to $\nabla$ and $\mathscr{K}_{D^2}$ is the $L^2$-kernel of $D^2$. In particular, when $\vbn=\Lambda_{\CC}^{k}T^*M$ (complexified version of $\Lambda^{k}T^*M$) and $D^2$ is the Hodge--deRham Laplacian on (complex-valued) $k$-forms, we recover theorem 5 of~\cite{Vieira-16} with $\rho\geq 0$.

In theorem~\ref{T:main-2} we accomplish the same goal as in theorem~\ref{T:main-1}, with the following hypotheses on $X$ and $V$: $|X|\leq\hat{a}\sqrt{\rho}$ with $0\leq \hat{a}<1$, $V-\div X\geq -a\rho-b$ with $0\leq a<1-\hat{a}$ and $b\geq 0$, and the condition $\lambda_1(M)>b/(1-a-\hat{a})$ (see~(\ref{E:first-ev}) for the definition of the bottom of the spectrum $\lambda_1(M)$ of the scalar Laplacian $\Del_{M}$). As a corollary, putting $X\equiv 0$ and specializing to (complex-valued) $k$-forms, we recover theorem 6 of~\cite{Vieira-16} with $\rho\geq 0$.

In theorems~\ref{T:main-3} and~\ref{T:main-4} we work in  the setting of a Clifford module $\vbn$ (over a complete Riemannian manifold $M$) equipped with a Clifford connection $\nabla$. We extend some vanishing results of~\cite{Zhou-20} to $D^2$, the square of the Dirac operator $D$ corresponding to $\nabla$, assuming weighted Poincar\'e inequality for $D$ (see~(\ref{E:poincare-2}) for precise formulation) and curvature conditions analogous to those in~\cite{Zhou-20}.

The paper is organized into five sections. After describing the notations, operators, and function spaces in section~\ref{S:res}, we state the main results in section~\ref{S:state-results}. The proofs of the main results are carried out in sections~\ref{S:pf-1-2} and~\ref{S:pf-3-4}.

\section{Description of Notations, Function Spaces, and Operators}\label{S:res}
\subsection{Basic Notations}\label{SS:s-2-1} In this paper we work in the setting of a connected Riemannian $n$-manifold $(M,g)$ without boundary. The symbol $d\nu_{g}$ denotes the volume measure on $M$: in local coordinates $x^1,x^2,\dots, x^n$, we have $d\nu_{g}=\sqrt{\det (g_{ij})}\,dx$, where $dx=dx^1\,dx^2\dots dx^n$ is the Lebesgue measure.

We use the notations $TM$, $T^*M$ for tangent and cotangent bundles of $M$ respectively. Additionally, $\Lambda^{k}T^*M$ stands for the $k$-th exterior power of the cotangent bundle $T^*M$, and $\Lambda_{\CC}^{k}T^*M$ indicates the complexified version of $\Lambda^{k}T^*M$. 


Throughout the paper, $\vbn \to M$ is a smooth Hermitian vector bundle over $M$ equipped with a Hermitian structure $\langle\cdot, \cdot\rangle$, linear in the first and antilinear in the second variable. We use $|\cdot|_{x}$ to indicate the fiberwise norms on $\vbn_{x}$, usually writing just $|\cdot|$ to simplify the notations. The symbols $\Gamma_{C^{\infty}}(\vbn)$ and $\vcomp$ denote smooth sections of $\vbn$ and smooth compactly supported sections of $\vbn$, respectively. In particular, for smooth complex-valued $k$-forms on $M$ we use the symbol $\Gamma_{C^{\infty}}(\Lambda_{\CC}^{k}T^*M)$, and for their compactly supported analogues, the symbol $\Gamma_{C_{c}^{\infty}}(\Lambda_{\CC}^{k}T^*M)$. When talking about complex-valued functions on $M$, the corresponding spaces will be indicated by $C^{\infty}(M)$ and $\mcomp$. Additionally, $C(M)$ denotes continuous (complex-valued) functions on $M$.

We will also use basic ``musical" isomorphisms coming from $g$: for a vector field $Y$ on $M$, the symbol $Y^{\flat}$ indicates the one-form associated to $Y$, while $\omega^{\sharp}$ refers to the vector field associated to the one-form $\omega$.

Lastly, we recall that the Levi--Civita connection $\nabla^{LC}$ has a unique metric extension to $\Lambda^{k}T^*M$, which we also denote by $\nabla^{LC}$.

\subsection{Description of $L^p$-spaces} For $1\leq p<\infty$, the notation $\Gamma_{L^p}(\vbn)$ refers to the space of $p$-integrable sections of $\vbn$ with the norm
\begin{equation}\label{E:lp-norm}
\|u\|^{p}_{p}:=\int_{M}|u(x)|_{x}^p\,d\nu_{g}(x),
\end{equation}
where $|\cdot|_{x}$ is the fiberwise norm in $\vbn_{x}$.


In the case $p=2$ we get a Hilbert space $\Gamma_{L^2}(\vbn)$ with the inner product

\begin{equation}\label{E:inner-mu}
(u,v)=\int_{M}\langle u(x),v(x)\rangle_{x}\,d\nu_{g}(x).
\end{equation}

For simplicity, we often drop the subscript $x$ from $\langle\cdot,\cdot\rangle_{x}$ and $|\cdot|_{x}$ and simply write $\langle\cdot,\cdot\rangle$ and $|\cdot|$.

For the $L^p$-space of (complex-valued) functions on $M$ we use the symbol $L^p(M)$, and in the formulas~(\ref{E:lp-norm}) and~(\ref{E:inner-mu}) we replace the fiberwise norm by the absolute value, and $\langle u(x),v(x)\rangle_{x}$ by $u(x)\overline{v(x)}$, where $\overline{z}$ indicates the conjugate of a complex number $z$.

\subsection{Covariant Schr\"odinger Operator}\label{SS:ornst-uhl-def}
With the basic function spaces in place, we turn to differential operators.
The first operator is $\nabla\colon \Gamma_{C^{\infty}}(\vbn)\to \Gamma_{C^{\infty}}(T^*M\otimes\vbn)$, a (smooth) metric covariant derivative on $\vbn$. Next, we have $\nabla^{\dagger}\colon \Gamma_{C^{\infty}}(T^*M\otimes\vbn)\to \Gamma_{C^{\infty}}(\vbn)$, the formal adjoint of $\nabla$ with respect to $(\cdot,\cdot)$, the inner product~(\ref{E:inner-mu}). Composing the latter two operators produces the so-called \emph{Bochner Laplacian} $\nabla^{\dagger}\nabla$. In the case of functions, we have the differential $d\colon C^{\infty}(M)\to \Gamma_{C^{\infty}}(\Lambda_{\CC}^{1}T^*M)$ and its formal adjoint $d^{\dagger}\colon \Gamma_{C^{\infty}}(\Lambda_{\CC}^{1}T^*M)\to C^{\infty}(M)$, understood with respect to the inner product $(\cdot,\cdot)$ in $L^2(M)$. The composition $d^{\dagger}d$, denoted by $\Delta_{M}$, is known as \emph{the scalar Laplacian} on $M$. We note that in our article $\nabla^{\dagger}\nabla$ and $\Delta_{M}$ are non-negative operators.

For a smooth vector field $Y$, we define the divergence of $Y$ as
\begin{equation}\label{E:div-def}
\div Y:=-d^{\dagger}(Y^{\flat}).
\end{equation}

Let $V\in \Gamma_{C^{\infty}}(\End \vbn)$ such that $V(x)\colon\vbn_x\to\vbn_x$ is a self-adjoint operator for all $x\in M$, and let $X$ be a smooth, real vector field on $M$.  We consider the \emph{covariant Schr\"odinger differential expression with potential $V$ and drift $X$}:
\begin{equation}\label{E:def-H}
H_{X,V}u:=\nabla^{\dagger}\nabla u+\nabla_{X}u+Vu.
\end{equation}
To make our terminology simpler, we refer to $H_{X,V}$ as \emph{covariant Schr\"odinger operator (with potential $V$ and drift $X$)}, understanding that the word ``operator" is used loosely.

\subsection{The Space $\mathscr{K}_{H_{X,V}}$}\label{SS:min-max-rel-1}
We define
\begin{equation}\label{E:KER-H-V}
\mathscr{K}_{H_{X,V}}:=\{u\in \Gamma_{L^2}(\vbn)\colon H_{X,V}u=0\},
\end{equation}
where $ H_{X,V}=0$ is understood in distributional sense, that is,
\[
(u, (H_{X,V})^{\dagger}v)=0,
\]
for all $v\in \Gamma_{C_{c}^{\infty}}(\vbn)$.

Since $V$ and $X$ are smooth and since $H_{X,V}$ is an elliptic operator, it follows (by local elliptic regularity) that $\mathscr{K}_{H_{X,V}}\subseteq \Gamma_{L^2}(\vbn)\cap \Gamma_{C^{\infty}}(\vbn)$. Thus, $\mathscr{K}_{H_{X,V}}$ is independent of any extension of $(H_{X,V})|_{\Gamma_{C_{c}^{\infty}(\vbn)}}$ to a closed operator in $\Gamma_{L^2}(\vbn)$.


\section{Statements of Results}\label{S:state-results}

In the first two theorems we assume that $M$ satisfies a weighted Poincar\'e inequality, which we describe as follows:

\subsection{Hypothesis (P1)}\label{assump_V_X_etc} Let $\rho\colon M\to \RR$. Assume that
\begin{itemize}
\item [(P1a)] $\rho$ is a continuous function such that $\rho(x)\geq 0$ for all $x\in M$;
\item [(P1b)] for all $f\in\mcomp$ we have
\begin{equation}\label{E:poincare-1}
\int_{M}\rho(x)|f(x)|^2\,d\nu_{g}(x)\leq \int_{M}|df(x)|^2d\nu_{g}(x),
\end{equation}
where $|\cdot|$ on the right hand side is the fiberwise norm in $T_{x}^*M$.
\end{itemize}

We are ready to state the first result.

\begin{theorem} \label{T:main-1} Assume that $M$ is a geodesically complete Riemannian manifold without boundary. Assume that $M$ satisfies the hypothesis (P1). Additionally, assume that one of the following two conditions is satisfied:
\begin{enumerate}
  \item [(m1)] $\rho(x)$ is not identically equal to $0$;
  \item [(m2)] the volume $\textrm{vol}(M)$ is infinite.
\end{enumerate}
Let $\vbn$ be a Hermitian vector bundle over $M$ equipped with a metric covariant derivative $\nabla$. Let $X$ be a smooth, real vector field on $M$ such that
\begin{equation}\label{E:hyp-X}
|X(x)|\leq \hat{a}\sqrt{\rho(x)},
\end{equation}
for all $x\in M$, where $0\leq \hat{a}<1$ is a constant, and $|\cdot|$ is the norm in $T_{x}M$.

Let $V\in \Gamma_{C^{\infty}}(\End \vbn)$ be a fiberwise self-adjoint endomorphism such that
\begin{equation}\label{E:v-rho}
V(x)-(\div X)I_{x}\geq -a\rho(x)I_{x},
\end{equation}
for all $x\in M$, where $0\leq a<1-\hat{a}$ is a constant, with $\hat{a}$ as in~(\ref{E:hyp-X}). (Here, $\div X$ is as in~(\ref{E:div-def}), and $I_{x}\colon \vbn_{x}\to \vbn_{x}$ is the identity endomorphism. The inequality~(\ref{E:v-rho}) is understood in quadratic-form sense in $\vbn_{x}$.)

Then, the set $\mathscr{K}_{H_{X,V}}$ from~(\ref{E:KER-H-V}) has the following property: $\mathscr{K}_{H_{X,V}}=\{0\}$.
\end{theorem}

As we will describe below, theorem~\ref{T:main-1}, in conjunction with Lichnerowicz--Weitzenb\"ock formula, leads to a vanishing result for the kernel of the square of the Dirac operator.

\subsection{Clifford Module}\label{SS:cb} In our description of the Clifford module we follow the conventions of chapter 10 in~\cite{Taylor}.
By a \emph{Clifford module} we mean a Hermitian vector bundle $\vbn$ over $M$ satisfying the following two properties:

\begin{enumerate}
  \item [(i)] each fiber $\vbn_x$ is a module over the Clifford algebra $C\ell(T^{*}_xM, g^{*}_{x})$, where $g^{*}_{x}$ is the co-metric on $T^{*}_xM$,
and
\[
\langle \xi \bullet u, v\rangle \ = \ \langle  u, \xi \bullet
v\rangle, \quad\text{for all }\xi\in T^{*}_xM\text{ with }|\xi|=1,\text{ and all }u,
v\in \vbn_x,
\]
where $\langle\cdot,\cdot\rangle$ is the fiberwise
inner product in $\vbn_x$ and ``$\bullet$" is the Clifford action.

  \item [(ii)] $\vbn$ is endowed with a metric
connection $\nabla$ satisfying the property
\[
\nabla_{X}(\omega\bullet s) \ = \ (\nabla^{LC}_{X}\omega)\bullet s +\omega\bullet
(\nabla_{X}s),
\]
for all $s\in \Gamma_{C^{\infty}}(\vbn)$, $\omega\in \Gamma_{C^{\infty}}(\Lambda^{1}T^*M)$, and $X\in \Gamma_{C^{\infty}}(TM)$, where
``$\bullet$" is the Clifford action and $\nabla^{LC}$ is the covariant derivative on $\Lambda^{1}T^*M$ induced from the Levi--Civita connection.
\end{enumerate}

The composition
\[
\Gamma_{C^{\infty}}(\vbn) \stackrel{\nabla} \longrightarrow
\Gamma_{C^{\infty}}(T^*M\otimes \vbn)\stackrel{\bullet} \longrightarrow
\Gamma_{C^{\infty}}(\vbn)\stackrel{i}\longrightarrow\Gamma_{C^{\infty}}(\vbn),
\]
where $i$ on the last arrow indicates the multiplication by $i=\sqrt{-1}$,
defines a first-order differential operator
\begin{equation}\label{E:Dirac}
D\colon\Gamma_{C^{\infty}}(\vbn)\to \Gamma_{C^{\infty}}(\vbn),
\end{equation}
called the \emph{Dirac operator} corresponding to the
Clifford module $(\vbn,\nabla)$.

\begin{remark}
We remind the reader that the authors of~\cite{lm-spin-book} use the convention $v \bullet v=-1\cdot q(v)$ for the Clifford algebra $C\ell(V,q)$ over a quadratic vector space $(V,q)$. Moreover, in their definition of the Clifford module, the authors of~\cite{lm-spin-book} use $C\ell(T_{x}M,g_{x})$ instead of $C\ell(T^{*}_{x}M,g^{*}_{x})$. As we follow the conventions from~\cite{Taylor}, the definitions of the Clifford module and the corresponding Dirac operator (see section~\ref{SS:cb} above) and certain formulas (for instance, the formulas~(\ref{E:Weitzenbock}),~(\ref{E:bw-curvature}) and~(\ref{E:ml-product}) below) look slightly different from their analogues in~\cite{lm-spin-book}.
\end{remark}

\begin{remark}\label{RR:lawson-spin} Since $M$ is geodesically complete, by theorem II.5.4 in~\cite{lm-spin-book}, it follows that $D|_{\vcomp}$ is an essentially self-adjoint operator in $\Gamma_{L^2}(\vbn)$ whose self-adjoint closure in $\Gamma_{L^2}(\vbn)$ we denote (again)
by $D$.

Defining
\begin{equation}\label{E:KER-D-squared}
\mathscr{K}_{D^2}:=\{u\in \Gamma_{L^2}(\vbn)\colon D^2u=0\}
\end{equation}
and using local elliptic regularity we see that $\mathscr{K}_{D^2}\subseteq \Gamma_{L^2}(\vbn)\cap \Gamma_{C^{\infty}}(\vbn)$.

Referring again to theorem II.5.4 in~\cite{lm-spin-book}, we have
\begin{equation}\label{E:KER-D-squared-1}
\mathscr{K}_{D^2}=\mathscr{K}_{D},
\end{equation}
where
\[
\mathscr{K}_{D}:=\{u\in \Gamma_{L^2}(\vbn)\colon Du=0\}
\]
\end{remark}

Before stating a corollary of theorem~\ref{T:main-1}, we recall a formula linking the Bochner Laplacian on a Clifford module $\vbn$ with the square $D^2$ of the corresponding Dirac operator $D$.

\subsection{Lichnerowicz--Weitzenb\"ock formula} In the setting of a Clifford module $(\vbn,\nabla)$ and the corresponding Dirac operator $D$, we have (see proposition 10.4.1 in~\cite{Taylor}):
\begin{equation}\label{E:Weitzenbock}
D^2u \ = \ \nabla^{\dagger}\nabla u+\mathscr{R}^{W}u,
\end{equation}
and $\mathscr{R}^{W}\in \Gamma_{C^{\infty}}(\End \vbn)$ is a fiberwise self-adjoint endomorphism.

More explicitly (see the formula (10.4.15) in~\cite{Taylor}), if $\{e_j\}_{j=1}^{n}$ is a local orthonormal frame field and $\{v_j\}_{j=1}^{n}$ is the corresponding dual frame (here, $n=\dim M$), then
\begin{equation}\label{E:bw-curvature}
\mathscr{R}^{W}u=-\frac{1}{2}\sum_{j,k=1}^{n}v_j\bullet v_k\bullet R^{\nabla}(e_j,e_k)u,
\end{equation}
for all $u\in \Gamma_{C^{\infty}}(\vbn)$, where $R^{\nabla}$ is the curvature tensor corresponding to the connection
$\nabla$.

For future reference, in this paper we call $\mathscr{R}^{W}$ \emph{Weitzenb\"ock curvature operator}.

The formula~(\ref{E:Weitzenbock}) and theorem~\ref{T:main-1} with $X\equiv 0$ and $V=\mathscr{R}^{W}$ lead to the following corollary:

\begin{cor}\label{C:main-1} Assume that $M$ is a geodesically complete Riemannian manifold without boundary. Assume that $M$ satisfies the hypothesis (P1). Additionally, assume that one of the following two conditions holds:
\begin{enumerate}
  \item [(i)] $\rho(x)$ is not identically equal to $0$;
  \item [(ii)] the volume $\textrm{vol}(M)$ is infinite.
\end{enumerate}
Let $\vbn$ be a Clifford module over $M$ equipped with a Clifford connection $\nabla$, and let $D$ be the associated Dirac operator. Assume that the Weitzenb\"ock curvature operator $\mathscr{R}^{W}$ satisfies the inequality
\begin{equation}\label{E:v-rho-weitz}
\mathscr{R}^{W}(x)\geq -a\rho(x)I_{x},
\end{equation}
for all $x\in M$, where $0\leq a<1$ is a constant. (Here, $I_{x}\colon\vbn_x\to \vbn_{x}$ is the identity endomorphism, and the inequality~(\ref{E:v-rho-weitz}) is understood in quadratic-form sense in $\vbn_{x}$.)

Then, the set $\mathscr{K}_{D^2}$ from~(\ref{E:KER-D-squared}) has the following property: $\mathscr{K}_{D^2}=\{0\}$.
\end{cor}

\begin{remark}\label{R:examples} Corollary~\ref{C:main-1} can be applied in the setting of a spin manifold $M$ and the associated spinor bundle $\vbn$, with $D$ being the so-called \emph{classical Dirac operator}. In this situation (see proposition 10.4.4 in~\cite{Taylor}), $\mathscr{R}^{W}= \textrm{scal}_{M}/4$ where $\textrm{scal}_{M}$ is the scalar curvature of $M$ (that is, the trace of the Ricci tensor). Corollary~\ref{C:main-1} can also be applied in the setting of $\vbn=\Lambda_{\CC}^{k}T^*M$ over an oriented Riemannian manifold $M$ (the complexified  version of $\Lambda^{k}T^*M$ with the corresponding Hermitian extension of the Riemannian structure and the corresponding extension of $\nabla^{LC}$). As explained in section 10.1 of~\cite{Taylor}, the bundle $\vbn=\Lambda_{\CC}^{k}T^*M$ (with its natural metric and connection $\nabla^{LC}$), has the structure of a Clifford module. In this situation, the associated Dirac operator $D$ is the so-called \emph{Gauss--Bonnet operator} operator $d+\delta$, where
\[
d\colon \Gamma_{C^{\infty}}(\Lambda_{\CC}^{k}T^*M)\to \Gamma_{C^{\infty}}(\Lambda_{\CC}^{k+1}T^*M),\quad \delta\colon \Gamma_{C^{\infty}}(\Lambda_{\CC}^{k}T^*M)\to \Gamma_{C^{\infty}}(\Lambda_{\CC}^{k-1}T^*M).
\]
are the standard differential and codifferential respectively.

In this setting, the operator $D^2$ becomes $D^2=d\delta+\delta d$, the so-called Hodge--deRham Laplacian acting on (complex-valued) $k$-forms, and the set
$\mathscr{K}_{D^2}$ from~(\ref{E:KER-D-squared}) is known as the space of \emph{$L^2$-harmonic complex-valued $k$-forms}.

Furthermore, in the setting $\vbn=\Lambda_{\CC}^{k}T^*M$, the operator $\mathscr{R}^{W}$ depends on the Riemannian curvature tensor of $M$, and, by proposition 10.4.2 in~\cite{Taylor}, in the case $\vbn=\Lambda_{\CC}^{1}T^*M$ we have $\mathscr{R}^{W}=\textrm{Ric}_{M}$, where $\textrm{Ric}_{M}$ is the Ricci tensor of $M$.

Thus, in the case of $\vbn=\Lambda_{\CC}^{k}T^*M$ and the Gauss--Bonnet operator $D=d+\delta$ , corollary~\ref{C:main-1} recovers theorem 5 from~\cite{Vieira-16} with $\rho\geq 0$, a vanishing result concerning $L^2$-harmonic complex-valued $k$-forms on $M$.
\end{remark}

Before stating the second theorem, we recall the concept of the bottom of the spectrum:

\subsection{Bottom of the Spectrum of $\Delta_M$} In the setting of a geodesically complete Riemannian manifold $M$, the operator $\Delta_{M}|_{\mcomp}$ is essentially self-adjoint in $L^2(M)$, with the corresponding self-adjoint closure denoted (for simplicity) again by $\Delta_{M}$. The \emph{bottom of the spectrum} of the self-adjoint operator $\Del_{M}$ is defined as
\begin{equation}\label{E:first-ev}
\lambda_1(M):=\inf_{f\in \mcomp}\frac{\int_{M}|df(x)|^2\,d\nu_{g}(x)}{\int_{M}|f(x)|^2d\nu_{g}(x)}
\end{equation}

\begin{remark}\label{R:lambda-1} By~(\ref{E:first-ev}) we have
\begin{equation}\label{E:first-ev-ineq}
\lambda_1(M)\int_{M}|f(x)|^2d\nu_{g}(x)\leq \int_{M}|df(x)|^2\,d\nu_{g}(x),
\end{equation}
for all $f\in\mcomp$, where $d\nu_{g}$ is the volume element on $M$ corresponding to the metric $g$.
\end{remark}

We now state the second theorem.

\begin{theorem} \label{T:main-2} Assume that $M$ is a geodesically complete Riemannian manifold without boundary. Assume that $M$ satisfies the hypothesis (P1).

Let $\vbn$ be a Hermitian vector bundle over $M$ equipped with a metric covariant derivative $\nabla$. Let $X$ be a smooth, real vector field on $M$ satisfying the condition~(\ref{E:hyp-X}).

Let $V\in \Gamma_{C^{\infty}}(\End \vbn)$ be a fiberwise self-adjoint endomorphism such that
\begin{equation}\label{E:v-rho-a-b}
V(x)-(\div X)I_{x}\geq -(a\rho(x)+b)I_{x},
\end{equation}
for all $x\in M$, where $0\leq a<1-\hat{a}$ and $b\geq 0$ are constants, with $0\leq \hat{a}<1$ as in~(\ref{E:hyp-X}). (Here, $\div X$ is as in~(\ref{E:div-def}), and $I_{x}\colon \vbn_x\to \vbn_{x}$ is the identity endomorphism. The inequality~(\ref{E:v-rho-a-b}) is understood in quadratic-form sense.)

Furthermore, assume that
\begin{equation}\label{E:f-eig-ab}
\lambda_1(M)>\frac{b}{1-a-\hat{a}},
\end{equation}
where $\lambda_1(M)$ is as in~(\ref{E:first-ev}) and $0\leq \hat{a}<1$ is as in~(\ref{E:hyp-X}).

Then, the set $\mathscr{K}_{H_{X,V}}$ from~(\ref{E:KER-H-V}) has the following property: $\mathscr{K}_{H_{X,V}}=\{0\}$.
\end{theorem}

The formula~(\ref{E:Weitzenbock}) and Theorem~\ref{T:main-2} with $X\equiv 0$ and $V=\mathscr{R}^{W}$ lead to the following corollary:

\begin{cor}\label{C:main-2} Assume that $M$ is a geodesically complete Riemannian manifold without boundary. Assume that $M$ satisfies the hypothesis (P1).

Let $\vbn$ be a Clifford module over $M$ equipped with a Clifford connection $\nabla$, and let $D$ be the associated Dirac operator. Assume that the inequality~(\ref{E:v-rho-a-b}) is satisfied with $X\equiv 0$ and the Weitzenb\"ock curvature operator $\mathscr{R}^{W}$ in place of $V$. Furthermore, assume that
\[
\lambda_1(M)>\frac{b}{1-a}.
\]
Then, the set $\mathscr{K}_{D^2}$ from~(\ref{E:KER-D-squared}) has the following property: $\mathscr{K}_{D^2}=\{0\}$.
\end{cor}

\begin{remark}\label{R:vieira-a-b} In the case of $\vbn=\Lambda_{\CC}^{k}T^*M$ and the Gauss--Bonnet operator $D=d+\delta$ (see remark~\ref{R:examples} for the notations), corollary~\ref{C:main-2} recovers theorem 6 from~\cite{Vieira-16} with $\rho\geq 0$, a vanishing result concerning $L^2$-harmonic complex-valued $k$-forms on $M$.
\end{remark}

For the remainder of this section, $\vbn$ is a Clifford vector bundle over $M$ equipped with a Clifford connection $\nabla$, and $D$ is the associated Dirac operator.

In the next two theorems we make the following assumption on $D$:

\subsection{Hypothesis (P2)}\label{SS:assump-P2} Let $\rho\colon M\to \RR$ be a continuous function. Assume that
\begin{itemize}
  \item [(P2a)] $\rho(x)\geq 0$ and $\rho(x)$ is not identically equal to $0$;
  \item [(P2b)] for all $u\in\vcomp$ we have
\begin{equation}\label{E:poincare-2}
\int_{M}\rho(x)|u(x)|^2\,d\nu_{g}(x)\leq \int_{M}|Du(x)|^2d\nu_{g}(x),
\end{equation}
where $|\cdot|$ is the fiberwise norm in $\vbn_{x}$.
\end{itemize}
\begin{theorem}\label{T:main-3} Assume that $M$ is a geodesically complete Riemannian manifold without boundary.
Let $\vbn$ be a Clifford module over $M$ equipped with a Clifford connection $\nabla$, and let $D$ be the associated Dirac operator. Assume that the hypothesis (P2) is satisfied.

Furthermore, assume that the Weitzenb\"ock curvature operator $\mathscr{R}^{W}$ satisfies the inequality
\begin{equation}\label{E:v-rho-new}
\mathscr{R}^{W}(x)\geq -a\rho(x)I_{x},
\end{equation}
for all $x\in M$, where $a\geq 0$ is a constant.  (Here, $I_{x}\colon\vbn_x\to \vbn_{x}$ is the identity endomorphism, and the inequality~(\ref{E:v-rho-new}) is understood in quadratic-form sense in $\vbn_{x}$.)

Then, the set $\mathscr{K}_{D^2}$ from~(\ref{E:KER-D-squared}) has the following property: $\mathscr{K}_{D^2}=\{0\}$.
\end{theorem}
\begin{remark}\label{R:zhou-a} In the case of $\vbn=\Lambda_{\CC}^{k}T^*M$ and the Gauss--Bonnet operator $D=d+\delta$ (see remark~\ref{R:examples} for the notations), theorem~\ref{T:main-3} recovers theorem 1.4 from~\cite{Zhou-20}, a vanishing result concerning $L^2$-harmonic complex-valued $k$-forms on $M$.
\end{remark}

\begin{theorem}\label{T:main-4} Assume that $M$ is a geodesically complete Riemannian manifold without boundary.
Let $\vbn$ be a Clifford module over $M$ equipped with a Clifford connection $\nabla$, and let $D$ be the associated Dirac operator. Assume that the hypothesis (P2) is satisfied. Furthermore, assume that the Weitzenb\"ock curvature operator $\mathscr{R}^{W}$ satisfies the inequality
\begin{equation}\label{E:v-rho-a-b-new}
\mathscr{R}^{W}(x)\geq -(a\rho(x)+b)I_{x},
\end{equation}
for all $x\in M$, where $a\geq 0$ and $b\geq 0$ are constants.  (Here, $I_{x}\colon\vbn_x\to \vbn_{x}$ is the identity endomorphism, and the inequality~(\ref{E:v-rho-a-b-new}) is understood in quadratic-form sense in $\vbn_{x}$.)

Furthermore, assume that
\begin{equation}\label{E:f-eig-ab-new}
\lambda_1(M)>b,
\end{equation}
where $\lambda_1(M)$ is as in~(\ref{E:first-ev}).

Then, the set $\mathscr{K}_{D^2}$ from~(\ref{E:KER-D-squared}) has the following property: $\mathscr{K}_{D^2}=\{0\}$.
\end{theorem}

\begin{remark}\label{R:zhou-a-b} In the case of $\vbn=\Lambda_{\CC}^{k}T^*M$ and the Gauss--Bonnet operator $D=d+\delta$ (see remark~\ref{R:examples} for the notations), theorem~\ref{T:main-4} recovers theorem 4.1 from~\cite{Zhou-20}, a vanishing result concerning $L^2$-harmonic complex-valued $k$-forms on $M$.
\end{remark}

\section{Proofs of Theorems~\ref{T:main-1} and~\ref{T:main-2}}\label{S:pf-1-2}

We begin with a description of the first-order $L^2$-type Sobolev spaces of (complex-valued) functions on $M$.

\subsection{Sobolev Space Notations} We define
\[
W^{1,2}(M):=\{v\in L^2(M)\colon dv\in \Gamma_{L^2}(\Lambda_{\CC}^{1}T^*M)\}.
\]
A local Sobolev space $W^{1,2}_{\loc}(M)$ consists of distributions $v$ on $M$ such that $\psi v\in W^{1,2}(M)$, for all $\psi\in\mcomp$. The space of compactly supported elements of $W^{1,2}_{\loc}(M)$ will be indicated by $W^{1,2}_{\comp}(M)$.

\begin{remark}\label{R:regularity} Let $\vbn$ be a Hermitian vector bundle over $M$. The following observation will be used in the sequel: if $u\in\Gamma_{C^{\infty}}(\vbn)$ then $|u|\in C(M)\cap W^{1,2}_{\loc}(M)$,
where $C(M)$ stands for continuous functions on $M$ and $|u(x)|$ is the fiberwise norm in $\vbn_{x}$.
\end{remark}

We also need a sequence of cut-off functions:

\subsection{Cut-Off Functions}\label{SS:cut-off} On a geodesically complete Riemannian manifold $M$ without boundary, there exists (see theorem III.3(a) in~\cite{Gue-book}) a sequence of functions $\chi_{k}\in\mcomp$ with the following properties:

\begin{enumerate}

\item [(c1)] for all $x\in M$, we have $0\leq \chi_{k}(x)\leq 1$;

\item [(c2)] for all compact sets $G\subset M$, there exists $n_0(G)\in\mathbb{N}$ such that for all $k>n_0$, we have $\chi_k|_{G}\equiv 1$;

\item [(c3)] $\displaystyle\sup_{x\in M}|d\chi_{k}(x)|\leq \frac{C}{k}$, where $C>0$ is a constant independent of $k$, and $|\cdot|$ is the fiberwise norm in $T_{x}^*M$.
\end{enumerate}
\begin{remark} From the property (c2) it follows that $\displaystyle\lim_{k\to \infty}\chi_{k}(x)=1$, for all $x\in M$.
\end{remark}

We now state a key lemma whose parts (ii) and (iii), in the presence of a vector field $X$, extend lemmas 1 and 2 from~\cite{Vieira-16}.

\begin{lemma}\label{L:V-DRIFT} Assume that $M$ is a geodesically complete Riemannian manifold without boundary. Let $\rho\colon M\to\RR$ be a continuous function such that $\rho(x)\geq 0$ for all $x\in M$. Furthermore, let $X$ be a smooth, real vector field on $M$ such that
\begin{equation}\label{E:hyp-X-1}
|X(x)|\leq \hat{a}\sqrt{\rho(x)},
\end{equation}
for all $x\in M$, where $0\leq \hat{a}<1$ is a constant.

Assume that $h\colon M\to\RR$ is a function belonging to $C(M)\cap W^{1,2}_{\loc}(M)\cap L^2(M)$  and satisfying the distributional inequality
\begin{equation}\label{E:distr-vieira}
h\Del_{M}h\leq -(Xh)h -(\div X)h^2+a\rho h^2+bh^2,
\end{equation}
where $a\geq 0$ and $b\geq 0$ are constants. (Here, $\Del_{M}$ is the non-negative Laplacian acting on  functions, and the notation $Xh$ means $dh(X)$, the action of $dh$ on $X$.)

Then, the following hold:

\begin{itemize}
\item [(i)] If $h$ satisfies~(\ref{E:distr-vieira}) with $X\equiv 0$, then
\begin{equation}\label{E:distr-vieira-c1}
\int_{M}|dh(x)|^2\,d\nu_{g}(x)\leq a\int_{M}\rho(x) h^2(x)\,d\nu_{g}(x)+b\int_{M}h^2(x)\,d\nu_{g}(x).
\end{equation}

\item [(ii)] Assume, in addition, that the hypothesis (P1) is satisfied. Furthermore, assume that $0\leq a<1-\hat{a}$, with $\hat{a}$ as in~(\ref{E:hyp-X-1}). Then,
\begin{equation}\label{E:distr-vieira-c2}
\int_{M}|dh(x)|^2\,d\nu_{g}(x)\leq \frac{b}{1-a-\hat{a}}\int_{M}h^2(x)\,d\nu_{g}(x).
\end{equation}

\item [(iii)] Assume, in addition, that the hypothesis (P1) is satisfied. Furthermore, assume that  $0\leq a<1-\hat{a}$, with $\hat{a}$ as in~(\ref{E:hyp-X-1}). Assume also that $h$ is not identically equal to $0$ and that $h$ satisfies~(\ref{E:distr-vieira}) with $b=0$. Then, $M$ has finite volume and $\rho$ is identically equal to $0$.
\end{itemize}
\end{lemma}

\begin{proof} The assertion (i) was proved in lemma 2.4 in~\cite{Zhou-20}. We remark that in lemma 2.4 of~\cite{Zhou-20} the author assumes $h\in C^{\infty}(M)$ and
\[
\int_{B(x_0,R)}h^2d\nu_{g}=o(R^2),
\]
as $R\to \infty$, where $B(x_0,R)$ is the geodesic open ball centered at $x_0\in M$ with radius $R$.

An inspection of the arguments used in the quoted lemma reveals that they work without changes under the hypothesis $h\in C(M)\cap L^2(M)\cap W^{1,2}_{\loc}(M)$.

We now prove the assertion (ii). As in lemmas 1 and 2 of~\cite{Vieira-16}, we use the integration-by-parts method, modified to account for the presence of the vector field $X$. Using the cut-off functions $\{\chi_k\}$ from section~\ref{SS:cut-off}, we multiply both sides of~(\ref{E:distr-vieira}) by $\chi_k^2$ and integrate each term over $M$.

In particular, remembering that the scalar Laplacian $\Delta_{M}w=d^{\dagger}dw$ is a non-negative operator and performing integration by parts on the left hand side of~(\ref{E:distr-vieira}) we have, after using the product rule on $d(\chi_k^2h)$,
\begin{equation}\label{E:dv-1}
\int_{M}|dh|^2\chi_k^2\,d\nu_{g}+2\int_{M}\langle hd\chi_k,\chi_kdh\rangle\,d\nu_{g}
\end{equation}
where $\langle\cdot,\cdot\rangle$ is the fiberwise inner product in $T_{x}^*M$.

Furthermore, performing integration by parts in the term with integrand $-(Xh)h\chi_k^2$ on the right hand side of~(\ref{E:distr-vieira}) and using the formula (see proposition 1.4 in appendix C of~\cite{Taylor})
\[
X^{\dagger}w=-Xw-(\div X)w,
\]
where $X^{\dagger}$ is the formal adjoint of the action of $X$ on a function $w$,
the right hand side of~(\ref{E:distr-vieira}) becomes
\begin{align}\label{E:dv-2}
&\int_{M}[X(\chi_k^2h)]h\,d\nu_{g}+\int_{M}(\div X)\chi_k^2h^2\,d\nu_{g}-\int_{M}(\div X)\chi_k^2h^2\,d\nu_{g}\nonumber\\
&+a\int_{M}\rho (\chi_kh)^2\,d\nu_{g}+b\int_{M} (\chi_kh)^2\,d\nu_{g}\nonumber\\
&=2\int_{M}(X\chi_k)\chi_kh^2\,d\nu_{g}+\int_{M}(Xh)\chi_k^2h\,d\nu_{g}\nonumber\\
&+a\int_{M}\rho (\chi_kh)^2\,d\nu_{g}+b\int_{M} (\chi_kh)^2\,d\nu_{g},
\end{align}
where we used the product rule
\[
X(\chi_k^2h)=2\chi_k(X\chi_k)h+\chi_k^2Xh.
\]
Remembering that~(\ref{E:dv-1}) is less than or equal to~(\ref{E:dv-2}), we obtain after some rearranging
\begin{align}\label{E:dv-3}
&\int_{M}|dh|^2\chi_k^2\,d\nu_{g}\leq -2\int_{M}\langle hd\chi_k,\chi_kdh\rangle\,d\nu_{g}\nonumber\\
&+2\int_{M}(X\chi_k)\chi_kh^2\,d\nu_{g}+\int_{M}(Xh)\chi_k^2h\,d\nu_{g}+a\int_{M}\rho (\chi_kh)^2\,d\nu_{g}\nonumber\\
&+b\int_{M} (\chi_kh)^2\,d\nu_{g}.
\end{align}

Our next goal is to use the hypotheses of part (ii) of the lemma to estimate (from above) the terms on the right hand side of~(\ref{E:dv-3}), and if, as a result of those estimates, we get terms with integrand $|dh|^2\chi_k^2$, make sure that those terms have as small coefficients as possible (with a total sum less than 1), so that after transferring those terms to the left hand side we a get a positive coefficient in front of the integral of $|dh|^2\chi_k^2$.

Before doing this we record a useful inequality for (real) numbers $\alpha$, $\beta$, and $\varepsilon>0$:
\begin{equation}\label{E:binomial}
\alpha\beta\leq\frac{\ve \alpha^2}{2}+\frac{\beta^2}{2\ve}.
\end{equation}

Using~(\ref{E:binomial}) and the (fiberwise) inequality (for one-forms $\omega$ and $\eta$)
\[
|\langle\omega,\eta\rangle|\leq |\omega||\eta|,
\]
we estimate the first term on the right hand side of~(\ref{E:dv-3}) as
\begin{align}\label{E:dv-4}
&-2\int_{M}\langle hd\chi_k,\chi_kdh\rangle\,d\nu_{g}\nonumber\\
&\leq \ve \int_{M}|dh|^2\chi_k^2\,d\nu_{g}+\ve^{-1}\int_{M}|d\chi_k|^2h^2\,d\nu_{g}.
\end{align}
Using the hypothesis~(\ref{E:hyp-X-1}), the inequality (here, $f$ is a function)
\[
|Xf|\leq |X||df|,
\]
and~(\ref{E:binomial}), we estimate the second term on the right hand side of~(\ref{E:dv-3}) as
\begin{align}\label{E:dv-5}
&2\int_{M}[X(\chi_k)]\chi_kh^2\,d\nu_{g}\leq 2\hat{a}\int_{M}\sqrt{\rho}|d\chi_k|\chi_kh^2\nonumber\\
&\leq \hat{a}\ve \int_{M}\rho(\chi_kh)^2\,d\nu_{g}+\hat{a}\ve^{-1}\int_{M}|d\chi_k|^2h^2\,d\nu_{g}.
\end{align}

Using the hypothesis~(\ref{E:hyp-X-1}) and the inequality~(\ref{E:binomial}) with $\ve=1$, we estimate the third term on the right hand side of~(\ref{E:dv-3}) as
\begin{align}\label{E:dv-6}
&\int_{M}(Xh)\chi_k^2h\,d\nu_{g} \leq \hat{a}\int_{M}\sqrt{\rho}|dh|\chi_k^2hd\nu_{g}\nonumber\\
&\leq \frac{\hat{a}}{2}\int_{M}\rho(\chi_kh)^2\,d\nu_{g}+\frac{\hat{a}}{2}\int_{M}\chi_k^2|dh|^2\,d\nu_{g}.
\end{align}

We keep the fourth and the fifth term on the right hand side of~(\ref{E:dv-3}) in their present form.

We now look at the right hand side of~(\ref{E:dv-3}) and the estimates~(\ref{E:dv-5}) and~(\ref{E:dv-6}). Adding the coefficients of the terms with integrand $\rho(\chi_kh)^2$, we get
\begin{equation}\label{E:coeff-temp}
\hat{a}\vep+2^{-1}\hat{a}+a.
\end{equation}

As $h\in W^{1,2}_{\loc}(M)$ and $\chi_k\in\mcomp$, we have $(\chi_k h)\in W^{1,2}_{\comp}(M)$. Thus, using Friedrichs mollifiers (with the help of a finite partition of unity, we may assume that $\chi_k h$ is supported in a coordinate neighborhood) and the hypothesis~(\ref{E:poincare-1}), it follows that
\begin{equation*}
\int_{M}\rho|\chi_k h|^2\,d\nu_{g}\leq \int_{M}|d(\chi_k h)|^2d\nu_{g}.
\end{equation*}

The latter inequality, together with the estimate,
\begin{align*}
&|d(\chi_k h)|^2=|\chi_kdh+hd\chi_k|^2\leq |\chi_kdh|^2+2|hd\chi_k||\chi_kdh|+|hd\chi_k|^2\nonumber\\
&\leq (1+\ve)\chi_k^2|dh|^2+(1+\ve^{-1})|d\chi_k|^2h^2,
\end{align*}
where we used~(\ref{E:binomial}), yield
\begin{equation}\label{E:poincare-tmp}
\int_{M}\rho|\chi_k h|^2\,d\nu_{g}\leq (1+\ve)\int_{M}\chi_k^2|dh|^2\nu_{g}+(1+\ve^{-1})\int_{M}|d\chi_k|^2h^2\,d\nu_{g}.
\end{equation}

We now go back to~(\ref{E:dv-3}), refer to the estimates~(\ref{E:dv-4}),~(\ref{E:dv-5}),~(\ref{E:dv-6}) and~(\ref{E:poincare-tmp}), remembering the coefficient~(\ref{E:coeff-temp}) in front of the sum of the terms with integrand $\rho(\chi_kh)^2$. As a result, we obtain, after moving all terms with integrand $\chi_k^2|dh|^2$ to the left hand side,
\begin{align}\label{E:dv-final}
&[1-\ve-2^{-1}\hat{a}-(\hat{a}\vep+2^{-1}\hat{a}+a)(1+\ve)]\int_{M}|dh|^2\chi_k^2\,d\nu_{g}\nonumber\\
&\leq[(\hat{a}\vep+2^{-1}\hat{a}+a)(1+\ve^{-1})+(\hat{a}+1)\ve^{-1}]\int_{M}|d\chi_k|^2h^2\,d\nu_{g}+b\int_{M} (\chi_kh)^2\,d\nu_{g}.
\end{align}

Since $0\leq a<1-\hat{a}$, we can choose a small enough $\ve>0$ so that
\begin{equation}\label{E:pos-coeff}
1-\ve-2^{-1}\hat{a}-(\hat{a}\vep+2^{-1}\hat{a}+a)(1+\ve)>0.
\end{equation}
Letting $k\to \infty$ in~(\ref{E:dv-final}) and using the properties of $\chi_k$ from section~\ref{SS:cut-off}, together with the assumption $h\in L^2(M)$, we get
\begin{align}\label{E:dv-final-1}
&[1-\ve-2^{-1}\hat{a}-(\hat{a}\vep+2^{-1}\hat{a}+a)(1+\ve)]\int_{M}|dh|^2\,d\nu_{g}\nonumber\\
&\leq b\int_{M} h^2\,d\nu_{g}.
\end{align}
Finally, letting $\ve\to 0$, we obtain~(\ref{E:distr-vieira-c2}).

We now prove the assertion (iii). With~(\ref{E:dv-final-1}) at our disposal, we can repeat the argument from the end of the proof of lemma 1 in~\cite{Vieira-16}. Putting $b=0$ in~(\ref{E:dv-final-1}) and keeping in mind~(\ref{E:pos-coeff}), we get
\[
\int_{M}|dh|^2\,d\nu_{g}\leq 0.
\]
This shows that there exists $\tilde{c}\in \RR$ such that $h(x)=\tilde{c}$ for all $x\in M$. By assumption in part (iii) of the lemma we have $\tilde{c}\neq 0$ and $h\in L^2(M)$. The only way the last sentence can be true is that $\textrm{vol}(M)$ be finite. Furthermore, using~(\ref{E:poincare-1}) with  $f=\chi_k$, we have
\[
\int_{M}\rho\chi_k^2\,d\nu_{g}\leq \int_{M}|d\chi_k|^2\,d\nu_{g}.
\]
Letting $k\to\infty$ in the latter inequality and using the properties of $\chi_k$ from section~\ref{SS:cut-off}, we obtain (remembering that $\textrm{vol}(M)$ is finite)
\[
\int_{M}\rho\,d\nu_{g}\leq 0,
\]
which, together with the hypothesis $\rho(x)\geq 0$, tells us that $\rho(x)=0$ for all $x\in M$.
This concludes the proof of assertion (iii) of the lemma.
\end{proof}

\subsection{Bochner Formula}
Here we recall a Bochner-type formula in the setting of a Hermitian vector bundle $\vbn$ over $M$ and a metric covariant derivative $\nabla$ on $\vbn$. For $u\in\Gamma_{C^{\infty}}(\vbn)$ we have
\begin{equation}\label{E:bochner-1}
\Delta_{M}\left(\frac{|u|^2}{2}\right)=\RE\langle\nabla^{\dagger}\nabla u,u\rangle-|\nabla u|^2,
\end{equation}
where $\Del_{M}$ is the non-negative Laplacian (acting on functions),
$\langle\cdot,\cdot\rangle$ is the fibrewise
inner product in $\vbn_x$, $|\cdot|$ on the left hand side is the norm in $\vbn_x$, and $|\cdot|$ on the right hand side is the norm in $(T^*M\otimes \vbn)_x$.

Using the formula
\begin{equation}\label{E:chain-rule-lap}
\Delta_{M}(f \circ w)= -f''(w)|dw|^2 +f'(w)\Delta_{M} w
\end{equation}
with real-valued functions $w\in W_{\loc}^{1,2}(M)$ and $f\in C^{\infty}(\RR)$,
we rewrite~(\ref{E:bochner-1}) as
\begin{equation}\label{E:bochner-2}
|u|\Delta_{M}|u|-|d|u||^2=\RE\langle\nabla^{\dagger}\nabla u,u\rangle-|\nabla u|^2.
\end{equation}

\subsection{Proof of Theorem~\ref{T:main-1}}

Starting with $u\in\mathscr{K}_{H_{X,V}}$, that is,
\[
\nabla^{\dagger}\nabla u=-\nabla_{X}u-Vu,
\]
and using~(\ref{E:bochner-2}), we obtain
\begin{align}\label{E:thm-1}
&|u|\Delta_{M}|u|=-\RE\langle \nabla_{X}u, u\rangle-\langle Vu, u\rangle\nonumber\\
&+|d|u||^2-|\nabla u|^2.
\end{align}
Keeping in mind that $X$ is real  and using the property
\begin{equation}\label{E:re-part}
X(|u|^2)=X\langle u, u\rangle=\langle \nabla_{X}u,u\rangle + \langle u,\nabla_{X}u\rangle=2\RE \langle \nabla_{X}u,u\rangle,
\end{equation}
together with the chain rule (here $Xf$ means $df(X)$),
\begin{equation}\label{E:re-part-chain}
X(|u|^2)=2|u|(X|u|),
\end{equation}
we can rewrite~(\ref{E:thm-1}) as
\begin{align}\label{E:thm-2}
&|u|\Delta_{M}|u|=-(X|u|)|u|-\langle Vu, u\rangle\nonumber\\
&+|d|u||^2-|\nabla u|^2.
\end{align}
Using the hypothesis~(\ref{E:v-rho}) and the so-called Kato's inequality (see formula (1.32) in~\cite{Eich-88} or formula (1) in~\cite{MH-00})
\begin{equation}\label{E:kato-fiber}
|d|u(x)||\leq |(\nabla u)(x)|,
\end{equation}
the formula~(\ref{E:thm-2}) leads to
\begin{equation*}
|u|\Delta_{M}|u|\leq -(X|u|)|u|-(\div X)|u|^2+a\rho|u|^2.
\end{equation*}

The last inequality and remark~\ref{R:regularity} tell us that the function $h(x):=|u(x)|$ satisfies the hypotheses of part (iii) of lemma~\ref{L:V-DRIFT}. Hence, looking at the conditions (m1) and (m2) of theorem~\ref{T:main-1}, we infer that $|u(x)|=0$ for all $x\in M$. This shows that $u=0$, that is, $\mathscr{K}_{H_{X,V}}=\{0\}$. $\hfill\square$

\subsection{Proof of Theorem~\ref{T:main-2}}

Starting with $u\in\mathscr{K}_{H_{X,V}}$ and arguing as in the proof of theorem~\ref{T:main-1}, we obtain~(\ref{E:thm-2}). The latter formula, together with the inequality~(\ref{E:kato-fiber}) and the hypothesis~(\ref{E:v-rho-a-b}), lead to
\begin{equation*}
|u|\Delta_{M}|u|\leq -(X|u|)|u|-(\div X)|u|^2+a\rho|u|^2+b|u|^2.
\end{equation*}
Referring to remark~\ref{R:regularity}, the last inequality tells us that the function $h(x):=|u(x)|$ satisfies the hypotheses of part (ii) of lemma~\ref{L:V-DRIFT}. Therefore, by~(\ref{E:distr-vieira-c2}) we have
\begin{equation}\label{E:dv-thm-2}
\int_{M}|d|u||^2\,d\nu_{g}\leq \frac{b}{1-a-\hat{a}}\int_{M}|u|^2\,d\nu_{g}.
\end{equation}

In particular, this estimate tells us $|u|\in W^{1,2}(M)$. As $M$ is geodesically complete, the space $\mcomp$ is dense in $W^{1,2}(M)$; see theorem 3.1 in~\cite{hebey}. Thus, the inequality~(\ref{E:first-ev-ineq}) holds with $|u|$ in place of $f$. Consequently, combining ~(\ref{E:first-ev-ineq}) and~(\ref{E:dv-thm-2}) yields
\begin{equation*}
\lambda_1(M)\int_{M}|u|^2\,d\nu_{g}\leq \frac{b}{1-a-\hat{a}}\int_{M}|u|^2\,d\nu_{g},
\end{equation*}
which upon rearranging leads to
\[
\left(\lambda_1(M)-\frac{b}{1-a-\hat{a}}\right)\int_{M}|u|^2\,d\nu_{g}\leq 0.
\]
The latter inequality and the hypothesis~(\ref{E:f-eig-ab}) lead to $|u(x)|=0$ for all $x\in M$, that is, $\mathscr{K}_{H_{X,V}}=\{0\}$. $\hfill\square$

\section{Proofs of Theorems~\ref{T:main-3} and~\ref{T:main-4}}~\label{S:pf-3-4}
In this section we work in the context of a Clifford module $\vbn$ over $M$, equipped with a Clifford connection $\nabla$ and the associated Dirac operator $D$. We first recall the product rule for $D$.

\subsection{Product Rule} By lemma II.5.5 in~\cite{lm-spin-book} (or exercise 10.1.1 in~\cite{Taylor}), for all $u\in \Gamma_{C^{\infty}}(\vbn)$  and  all $\psi\in C^{\infty}(M)$, we have
\begin{equation}\label{E:ml-product}
D(\psi u)=i(d\psi)\bullet u+\psi Du,
\end{equation}
where ``$\bullet$" is the Clifford multiplication and $i=\sqrt{-1}$.

The following lemma is a Dirac-operator analogue of lemma 2.2 from~\cite{Zhou-20}.

\begin{lemma}\label{L:zhou-lemma} Assume that $M$ is a geodesically complete Riemannian manifold without boundary.
Let $\vbn$ be a Clifford module over $M$ equipped with a Clifford connection $\nabla$, and let $D$ be the associated Dirac operator. Assume that $D$ satisfies the hypothesis (P2). Furthermore, assume that $u\in \Gamma_{C^{\infty}}(\vbn)\cap \Gamma_{L^2}(\vbn)$ is a solution of the equation $Du=0$. Then,
\begin{equation}\label{E:zhou-lemma-1}
\int_{M}\rho|u|^2\,d\nu_{g}\leq 0.
\end{equation}
\end{lemma}
\begin{proof} Let $\{\chi_k\}$ be as in section~\ref{SS:cut-off} and let $u\in \Gamma_{C^{\infty}}(\vbn)$. Then, $\chi_k u\in\vcomp$, and we can use~(\ref{E:poincare-2}) to get
\begin{align}\label{E:zhou-w-1}
&\int_{M}\rho\chi_k^2|u|^2\,d\nu_{g}\leq \|D(\chi_ku)\|_{2}^2\nonumber\\
&=\|i(d\chi_k)\bullet u+\chi_k Du\|_{2}^2=\|i(d\chi_k)\bullet u\|_{2}^2\leq \frac{C}{k^2}\|u\|_{2}^2,
\end{align}
where $C>0$ is a constant and $\|\cdot\|_{2}$ is the norm in $\Gamma_{L^2}(\vbn)$. Here, in the first equality we used~(\ref{E:ml-product}), in the second equality we used the assumption $Du=0$, and in the third inequality we used the property (c3) from section~\ref{SS:cut-off}.
Letting $k\to\infty$ in~(\ref{E:zhou-w-1}) we obtain~(\ref{E:zhou-lemma-1}).
\end{proof}

Our last ingredient is a geometric formula.

\subsection{Bochner-Weitzenb\"ock Formula} Combining~(\ref{E:bochner-2}) and~(\ref{E:Weitzenbock}) we get
\begin{align}\label{E:bochner-3}
&|u|\Delta_{M}|u|-|d|u||^2=\RE\langle D^2 u,u\rangle-\RE\langle\mathscr{R}^{W}u,u\rangle-|\nabla u|^2\nonumber\\
&=\RE\langle D^2 u,u\rangle-\langle\mathscr{R}^{W}u,u\rangle-|\nabla u|^2,
\end{align}
where in the second equality we used fiberwise self-adjointness of  $\mathscr{R}^{W}$.

We are now ready to prove theorem~\ref{T:main-3}.

\subsection{Proof of Theorem~\ref{T:main-3}}

Starting with $u\in\mathscr{K}_{D^2}$ and using~(\ref{E:bochner-3}), we get
\begin{align}\label{E:thm-3}
&|u|\Delta_{M}|u|=-\langle \mathscr{R}^{W} u, u\rangle+|d|u||^2-|\nabla u|^2\nonumber\\
&\leq a\rho|u|^2,
\end{align}
where the last estimate (with $a\geq 0$) follows from the hypothesis~(\ref{E:v-rho-new}) and the inequality~(\ref{E:kato-fiber}).

Looking at the inequality~(\ref{E:thm-3}) and recalling remark~\ref{R:regularity} we can see that the function $h(x):=|u(x)|$ satisfies the hypotheses of part (i) of lemma~\ref{L:V-DRIFT} with $b=0$. Therefore, appealing to~(\ref{E:distr-vieira-c1}) with $b=0$, we get
\begin{equation*}
\int_{M}|d|u||^2\,d\nu_{g}\leq a\int_{M}\rho |u|^2\,d\nu_{g},
\end{equation*}
which in combination with lemma~\ref{L:zhou-lemma} (remember~(\ref{E:KER-D-squared-1}), that is, $Du=0$) yields
\begin{equation*}
\int_{M}|d|u||^2\,d\nu_{g}\leq a\int_{M}\rho |u|^2\,d\nu_{g}\leq 0.
\end{equation*}
Thus, there exists a number $\tilde{c}\geq 0$ such that $|u(x)|=\tilde{c}$ for all $x\in M$. Assume for a moment that $\tilde{c}>0$.

Since (see hypothesis (P2a)) the function $\rho$ is not identically equal to $0$, we have
\[
\int_{M}\rho\,d\nu_{g}>0.
\]
On the other hand~(\ref{E:zhou-lemma-1}) yields
\[
0\geq \int_{M}\rho|u|^2\,d\nu_{g}=\tilde{c}\int_{M}\rho\,d\nu_{g},
\]
that is (since we assumed $\tilde{c}>0$),
\[
\int_{M}\rho\,d\nu_{g}\leq 0.
\]
The obtained contradiction says that $\tilde{c}$ must equal $0$, that is, $\mathscr{K}_{D^2}=\{0\}$. $\hfill\square$

\subsection{Proof of Theorem~\ref{T:main-4}}

Starting with $u\in\mathscr{K}_{D^2}$ and using~(\ref{E:bochner-3}), we get
\begin{align}\label{E:thm-4}
&|u|\Delta_{M}|u|=-\langle \mathscr{R}^{W} u, u\rangle+|d|u||^2-|\nabla u|^2\nonumber\\
&\leq a\rho|u|^2+b|u|^2,
\end{align}
where the last estimate (with $a\geq 0$) follows from the hypothesis~(\ref{E:v-rho-a-b-new}) and the inequality~(\ref{E:kato-fiber}).

The estimate~(\ref{E:thm-4}) and remark~\ref{R:regularity} tell us that the function $h(x):=|u(x)|$ satisfies the hypotheses of part (i) of lemma~\ref{L:V-DRIFT}. Therefore~(\ref{E:distr-vieira-c1}) gives
\begin{equation*}
\int_{M}|d|u||^2\,d\nu_{g}\leq a\int_{M}\rho |u|^2\,d\nu_{g}+b\int_{M}|u|^2\,d\nu_{g}.
\end{equation*}
Remembering~(\ref{E:KER-D-squared-1}), that is, $Du=0$, and using lemma~\ref{L:zhou-lemma}, the last estimate leads to
\begin{equation*}
\int_{M}|d|u||^2\,d\nu_{g}\leq b\int_{M}|u|^2\,d\nu_{g}.
\end{equation*}
From hereon, we use~(\ref{E:f-eig-ab-new}) and argue in the same way as in the last stage of the proof of theorem~\ref{T:main-2} to infer that $\mathscr{K}_{D^2}=\{0\}$. $\hfill\square$

\end{document}